\documentclass[11pt]{amsart}

\usepackage{amsmath, amsthm, amssymb, amsfonts, enumerate}
\usepackage[colorlinks=true,linkcolor=blue,urlcolor=blue]{hyperref}
\usepackage{color}
\usepackage{geometry}

\newtheorem{theorem}{Theorem}[section]
\newtheorem{remark}[theorem]{Remark}
\newtheorem{assumption}[theorem]{Assumption}

\newtheorem{proposition}[theorem]{Proposition}
\newtheorem{corollary}[theorem]{Corollary}

\def \cB{{\mathcal B}}
\def \cC{{\mathcal C}}
\def \cF{{\mathcal F}}

\def \cP{{\mathcal P}}
\def \E{\mathsf{E}}
\def \P{\mathsf{P}}
\def \R{\mathbb{R}}

\def \G{\Gamma}
\DeclareMathOperator*{\esssup}{ess\,sup}
\def \mathds {{\bf}}

\title[Reflected BSDEs with interconnected obstacles]{A note on a new existence result for reflected BSDEs with interconnected obstacles}
\author[De Angelis, Ferrari, Hamad\`ene]{Tiziano De Angelis \and Giorgio Ferrari
\and Sa\"id Hamad\`ene}
\keywords{}
\address{T.~De Angelis: School of Mathematics, University of Leeds, Woodhouse Lane, LS2 9JT Leeds, UK.}
\email{\href{mailto:t.deangelis@leeds.ac.uk}{t.deangelis@leeds.ac.uk}}
\address{G.~Ferrari: Center for Mathematical Economics, Bielefeld University, Universit\"atsstrasse 25, 33615, Bielefeld, Germany.}
\email{\href{mailto:giorgio.ferrari@uni-bielefeld.de}{giorgio.ferrari@uni-bielefeld.de}}
\address{S.~Hamad\`ene: Le Mans Universit\'e, LMM, Avenue Olivier Messiaen, 72085 Le Mans, Cedex 9, France.}
\email{\href{mailto:hamadene@univ-lemans.fr}{hamadene@univ-lemans.fr}}
\date{\today}

\numberwithin{equation}{section}

\begin{document}

\begin{abstract}
In this note we prove existence of a solution to a system of Markovian BSDEs with interconnected obstacles. A key feature of our system, and the main novelty of this paper, is that we allow for the driver $f_i$ of the $i$-th component of the $Y$-process to depend on all components of the $Z$-process. This extends the existing theory on reflected BSDEs, which only addresses problems where $f_i$ depends on $Z^i$.
\end{abstract}

\maketitle

\section{Introduction}
\label{introduction}

In this note we study existence of a solution of a system of reflected backward stochastic differential equations (BSDEs) with inter-connected obstacles. Letting $T>0$ and $t\in [0,T]$, the problem is to find $m$ trebles of $(\cF_s)_{s\in [t,T]}$-adapted processes $(Y^i,Z^i,K^i)_{i\in \G}$, where $\G:=\{1,\dots,m\}$, $Y^i,\,K^i\in\R$ and $Z^i\in\R^d$, $d\ge1$, such that for any $i\in \G$ we have: $\forall s\in [t,T]$,
 \begin{equation}\label{main-systemintro}
\left\{
\begin{array}{l}
{Y}^{i}_s=h_i(X^{t,x}_T)+\int^T_sf_i(r,X^{t,x}_r,({Y}_r^{k})_{k\in \G},({Z}^{k}_r)_{k\in \G})dr+{K}^{i}_T-{K}^{i}_s-\int^T_s{Z}^{i}_rdB_r\,\\[+4pt]
      {Y}^{i}_s\geq \max\limits_{j\neq i}\{{Y}^{j}_s-{g}_{ij}(s,X^{t,x}_s)\}\\[+4pt]
\int_t^T( {Y}^{i}_s-\max\limits_{j\neq i}\{{Y}^{j}_s-{g}_{ij}(s,X^{t,x}_s)\})dK^i_s=0
\end{array}
\right.
\end{equation}
where:
\begin{itemize}
\item[ i)] $B$ is a $d$-dimensional Brownian motion and we denote $Z^i=(Z^{i1},Z^{i2}\ldots Z^{id})$ and ${Z}^{i}dB:=\sum_{j=1}^dZ^{ij}dB^j$;
\item[ii)]  for any $i,j\in \G$, the functions $h_i$, $f_i$ and $g_{ij}$ are deterministic;
\item[iii)] for any $(t,x)\in [0,T]\times\R^k$, the process $X^{t,x}$ is solution of the following SDE:
\[
X^{t,x}_{s}=x+\int_t^{s}b(r,X^{t,x}_r)dr+\int_t^{s}\sigma(r,X^{t,x}_r)dB_r,\quad t\le s\leq T.
\]
\end{itemize}
Since randomness in \eqref{main-systemintro} stems from the Markov process $X^{t,x}$, we say that the system \eqref{main-systemintro} is Markovian.

If for $i=1,...,m$, $f_i$ does not depend on $(y^i)_{i=1,m}$ and $(z^i)_{i=1,m}$, the solution of \eqref{main-systemintro} is linked to an optimal switching problem. The latter is a problem in which a decision maker (or controller) controls a (stochastic) system which may operate in different modes (e.g., a power plant). The aim of the controller is to maximise some performance criterion by optimally choosing controls of the form $\delta:=(\tau_n,\zeta_n)_{n\geq 0}$. Here $(\tau_n)_{n\geq 0}$ denotes an increasing sequence of (stopping) times at which the controller switches the system across different operating modes. Moreover, $(\zeta_n)_{n\ge 0}$ is a sequence of random variables taking their values in $\{1,...,m\}$. Each $\zeta_n$ represents the system's new operating mode after a switch has occurred at time $\tau_n$.

In this setting it is well known (see e.g. \cite{cek, dhp,hamadene-morlais,hutang}, etc.) that $Y^{i}_t$ is the \emph{value} of an optimal switching strategy, i.e., given $\tau_0=t$ and $\zeta_0=i$, it holds
\begin{equation}\label{carac:intro}
Y^{i}_t=\esssup_{\delta:=(\tau_n,\zeta_n)_{n\geq 0}}\E\Big[\int_t^T f_{a_s}(s, X^{t,x}_s)ds-A^\delta_T+h_{a_T}(X^{t,x}_T)\Big|\cF_t\Big]
\end{equation}
where the process $a:=(a_s)_{s\leq T}$ is indicating the mode of the
system at time $s$, $A^\delta_T$ stands for the total switching cost
when the strategy $\delta$ is implemented and, finally,
$h_{a_T}(X^{t,x}_T)$ is the terminal payoff. It is also known that
the solution of \eqref{main-systemintro} enables to construct an
optimal strategy as well.

It is important to remark that a characterization as in
\eqref{carac:intro} also holds in non-Markovian frameworks and
we mention that switching problems often arise in economics, finance
and power system management, amongst many other applied fields (see
e.g. \cite{[BO], [BO1], [BS],[CL],
[DX],[DP],[DZ2],porchet,shi,tri1,tri} and the references therein).

Problems like \eqref{main-systemintro} have been studied at a
theoretical level in the case when, for any $i=1,...,m$, the
function $f_i$ depends only on the state variable $z^i$ and possibly
on $(y^i)_{i\in \G}$ (see, e.g., \cite{cek,hamadene-morlais}). In
that setting existence (and uniqueness) results were provided (also
for the non-Markovian case) by using comparison principles for solutions of BSDEs. Such comparisons do not hold in our framework since $f_i$ depends on $(z^i)_{i\in \Gamma}$, hence we must rely on different methods.

The main objective of this paper is indeed to consider systems in which, for $i=1,...,m$, functions $f_i$ not only depend on the state variable $z^i\in\R^d$ but on all components of the state variable $z:=(z^i)_{i\in\G}$.
In particular we show that if $\sigma \sigma^\top$ is bounded and uniformly elliptic, then \eqref{main-systemintro} has a solution, provided that the switching costs $(g_{ij})_{i,j\in \G}$ are sufficiently regular. We adopt a usual penalization scheme (see \eqref{eq:penalisedBSDE} below) to handle the reflection constraints and rely deeply on essentially three facts: i) the representation of solutions of BSDEs as deterministic functions of $t$ and $X$; ii) smoothness of $g_{ij}$, which enables fundamental bounds in the penalisation scheme; iii) existence of a transition density of $X^{t,x}_s$ for any $s>t$, which satisfies a so-called \emph{domination condition}.

Our work is a first step towards the solution of \eqref{main-systemintro} in general non-Markovian setup. The paper is organized as follows. In Section \ref{sec:problemformulation} we set out the notations and make standing assumptions that hold throughout the paper. In Section \ref{sec:mainresult}, we prove our existence result in a number of steps. First we introduce the penalization scheme associated with \eqref{main-systemintro} and study its properties (in particular we show in Proposition \ref{prop:penalisation} that the time derivative of the penalizing term is uniformly bounded). Then we use an argument based on weak convergence and the aforementioned domination condition (see also \cite{HLP}) to obtain a convergent subsequence of solutions to the penalized problems. We finally show that the limit of such subsequence solves \eqref{main-systemintro} and provide a representation of $(Y^i,Z^i)_{i\in\G}$ as deterministic functions of $(t,X)$. We leave for future investigation questions of uniqueness of the solution and its links to optimal switching problems. The latter will inevitably feature a  more general structure than \eqref{carac:intro}.


\section{Setting and problem formulation}
\label{sec:problemformulation}

\subsection{Setting}
\label{sec:setting}

Let $T$ be a fixed positive real constant, and let $(\Omega,\cF,\P)$ be a probability space on which we define a $d$-dimensional standard Brownian motion $B:=(B_t)_{t\in [0,T]}$. For $t\leq T$, we set $\cF^\circ_t:=\sigma\{B_s, s\le t\}$, the $\sigma$-algebra generated by $B$ up to time $t$, and we denote by $(\cF_t)_{t\leq T}$ the completion of $(\cF^\circ_t)_{t\leq T}$ with the $\P$-null sets of $\cF$. For arbitrary integer numbers $d\ge1$ and $m\ge 1$, we denote by $|\cdot|_{d}$ and $|\cdot|_{m\times d}$ the Euclidean norms in $\R^d$ and $\R^{m\times d}$, respectively. Occasionally, when no confusion may arise, we will simplify our notation using $|\cdot|$ for either $|\cdot|_{d}$ or $|\cdot|_{m\times d}$.
Next, we introduce the following sets:
\begin{itemize}
\item[(i)] $\cP$ is the $\sigma$-algebra of $\cF_t$-progressively measurable sets of $\Omega\times[0,T]$;
\item[(ii)] $\cB(\R^d)$ is the Borel $\sigma$-algebra on $\R^d$, $d \geq 1$;
\item[(iii)] $\textbf{H}^2_T(\R^d):= \{\zeta:=(\zeta_t)_{t\leq T}$ is a $\R^d$-valued, ${\cP}$-measurable process such that $\E\big[\int^T_0{\vert {\zeta_t}\vert}^2 dt\big]<\infty\}$;
\item[(iv)]$\textbf{S}^2_T(\R):= \{\xi:=(\xi_t)_{t\leq T}$ is a $\R$-valued, ${\cP}$-measurable, continuous process such that $\E\big[\sup_{0\leq t\leq T}{\vert{\xi_t}\vert}^2\big]<\infty\}$;
\item[(v)] $\textbf{A}^2_T$ is the subspace of $\textbf{S}^2_T$ of non-decreasing processes which are null at $t=0$.
\item[(vi)] $\mathcal{C}^{1,2}([0,T] \times \R^d)$ (or simply $\mathcal{C}^{1,2}$) is the set of real-valued functions defined on $[0,T] \times \R^d$ which are once continuously differentiable in $t$ and twice continuously differentiable in $x$.
\end{itemize}

Let $X:=(X_s)_{s\leq T}$ be an $(\cF_t)_{t\le T}$-Markov process, valued in $\R^k$, $k\geq1$. For $(t,x)\in[0,T]\times\R^k$ fixed, we denote by $X^{t,x}$ the process $(X_s)_{s\in [t,T]}$ such that $\P(X^{t,x}_t=x)=1$, and by $\mu(t,x;s,dy)$ the law of $X^{t,x}_s$ (for $s\ge t$), i.e., $\P(X^{t,x}_s\in A)=\mu(t,x;s,A)$ for any $A\in \cB(\R^k)$. We now introduce the following condition on the Markov process $X$.
\medskip

\noindent{\bf (A0)} [\emph{${L}^2$-domination condition}]. We say that the process $X$ satisfies the ${L}^2$-domination condition if the family of laws $\{\mu(t,x;s,dy),\ s\in [t,T],\,t\in[0,T],\,x\in\R^k\}$ verifies the following condition: There exists $x_0\in\R^k$ such that, for any $t\in[0,T]$ and $x\in \R^k$ and any $\delta>0$ (such that $\delta+t\leq T$) there exists an application $\phi_{t,x,x_0}^{\delta}: [t,T] \times \R^k \mapsto \R_+$ with the following properties:
\begin{itemize}
\item[(a)] $\mu(t,x;s,dy)ds=\phi_{t,x,x_0}^{\delta}(s,y)\mu(0,x_0;s,dy)ds$ for all $(s,y)\in [t+\delta,T]\times\R^k;$
\item[(b)] $\forall N\geq 1,\ \phi_{t,x,x_0}^{\delta}\in {L}^{2}([t+\delta,T]\times [-N,N]^k;\ \mu(0,x_0;s,dy)ds)$.
\end{itemize}
\medskip

\noindent \textbf{Example}. A Markov process fulfilling the
${L}^2$-domination condition is given by the solution of the
stochastic differential equation
\begin{equation}
\label{cmSDE}
X^{t,x}_s=x+\int_t^s b(r,X^{t,x}_r)dr+\int^s_t\sigma(r,X^{t,x}_r)dB_r,\,\,~s\in [t,T],
\end{equation}
with $(t,x)\in[0,T]\times\R^k$, under the conditions detailed below:
\begin{itemize}
\item[(E1)] We take $k=d$ (recall that $B$ is $d$-dimensional), and the functions $b: [0,T]\times \R^d \mapsto \R^d$ and $\sigma: [0,T]\times \R^d \mapsto \R^{d\times d}$ are jointly continuous in $(t,x)$. Moreover they are Lipschitz continuous in $x$, uniformly with respect to $t$, i.e.~there exists a non-negative constant $C_1$ such that for any $(t,x,x')\in [0,T]\times \R^{d+d}$ we have
\begin{align}\label{eq:Lip0}
|\sigma(t,x)-\sigma(t,x')|_{d\times d}+|b(t,x)-b(t,x')|_d\leq C_1|x-x'|_d.
\end{align}
The above property, together with the joint continuity, imply that
$b$ and $\sigma$ have sub-linear growth in $x$, i.e.~there
is $C_2>0$ such that
\begin{align}\label{eq:sublin}
|b(t,x)|_d+|\sigma(t,x)|_{d\times d}\leq C_2(1+|x|_d).
\end{align}

\item[(E2)]We assume further that $\sigma\sigma^\top$ is uniformly elliptic, i.e., that there exists a constant $\theta>0$ such that for any $(t,x) \in [0,T] \times \R^d$ (denoting by $\langle\cdot,\cdot\rangle_d$ the scalar product in $\R^d$) it holds
$$\theta^{-1}|\zeta|_d^2\leq \langle\sigma(t,x)\sigma(t,x)^\top\zeta,\zeta\rangle_d \leq \theta |\zeta|_d^2,\,\,\zeta \in \R^d.$$
\end{itemize}

Condition $(E1)$ guarantees that the solution of \eqref{cmSDE} exists and it is unique (see, e.g., Chapter 5 of \cite{KS} for more details). Moreover $(E2)$ implies that $\sigma$ is bounded and invertible, with bounded inverse $\sigma^{-1}$. Uniform ellipticity of $\sigma$ also implies (cf.\ \cite{aronson}) that for any $(t,x) \in [0,T] \times \R^d$ the law $\mu(t,x;s,dy)$ of $X^{t,x}_s$ has a density function $p(t,x;s,y)$ such that for every $s>t$ and $y\in \R^d$
\begin{equation}
\label{Aronson}
m(s-t)^{-\frac{d}{2}}\exp\Big\{-\frac{\Lambda|y-x|_d^2}{s-t}\Big\}\leq p(t,x;s,y)\leq M(s-t)^{-\frac{d}{2}}\exp\Big\{-\frac{\lambda|y-x|_d^2}{s-t}\Big\}.
\end{equation}
Here $m$, $M$, $\lambda$ and $\Lambda$ are positive constants such that $m\leq M$ and $\lambda\leq \Lambda$. It is then easily verified that the family $\{\mu(t,x;s,dy),\ s\in [t,T],\,t\in[0,T],\,x\in\R^k\}$ satisfies the $L^2$-domination condition (A0).
\medskip

For future reference we also recall that $(E1)$ above implies that
\begin{equation}
\label{estimx}
\E\big[\sup_{t\le s\leq T}|X^{t,x}_s|_d^{\gamma}\big]\leq C(1+|x|_d^{\gamma}),
\end{equation}
for any $\gamma \geq 1$ and with $C=C(T,\gamma, C_2)>0$, independent of $x$. Moreover, the infinitesimal generator of $X^{t,x}$, denoted by $\mathbb{L}_X$, reads
\begin{align}\label{eq:LX}
\big(\mathbb{L}_X\psi\big)(x)=\frac{1}{2}\sum_{i,j=1}^d((\sigma\sigma^\top)_{ij}\partial^2_{x_ix_j}\psi)(x)+
\sum_{i=1}^d(b_i\partial_{x_i}\psi)(x),
\end{align}
for $\psi \in \cC^{2}(\R^d)$ and for any $x \in \R^d$.

At this point it is worth noticing that the results of this paper hold for a general Markov process $X$ provided that $X$ is a semi-martingale, it satisfies the ${L}^2$-domination condition and \eqref{estimx},
and the increments of the bounded variation part of the processes $(g_{ij}(t,X_t))_{t\in[0,T]}$ are non-positive (see Assumption \textbf{(A2)}-(b) below). However, in order to avoid technicalities and to improve readability of the paper, from now on we make the following standing assumption
\begin{assumption}\label{ass:1}
We assume that $k=d$ and that $X^{t,x}$ is the solution of \eqref{cmSDE} under conditions (E1) and (E2) above, hence satisfying the ${L}^2$-domination condition \textbf{(A0)}.
\end{assumption}

\subsection{A system of reflected BSDEs with interconnected obstacles}
\label{sec:coeff}

Here we formulate the problem object of our study, i.e.~a system of reflected BSDEs with interconnected obstacles. We begin by introducing $\G:=\{1,2,...,m\}$ and functions $(f_i)_{i\in \G}$, $(h_i)_{i\in \G}$ and $(g_{ij})_{i,j\in \G}$ which satisfy the requirements below.
\medskip

\noindent {\bf(A1)} For any $i\in \G$, the function
\[
f_i:\,\,(t,x,(y_k)_{k\in \G},(z_k)_{k\in \G})\in [0,T]\times \R^{d+m+m\times d}\longmapsto f_i(t,x,(y_k)_{k\in \G},(z_k)_{k\in \G}) \in \R
\]
\begin{itemize}
\item[(a)] is Lipschitz continuous in the variables $(\vec{y},\vec{z}):=((y_k)_{k\in \G},(z_k)_{k\in \G})$, uniformly with respect to $(t,x)$; that is, there is $C>0$ such that
\begin{equation}
\label{growthcondition}
|f_i(t,x,\vec{y}_1,z_1)-f_i(t,x,\vec{y}_2,\vec{z_2})|\leq C(|\vec{y}_1-\vec{y}_2|_m+|\vec{z_1}-\vec{z_2}|_{m\times d}),
\end{equation}
for any $(t,x)\in [0,T]\times \R^d$, $(\vec{y}_1,\vec{y}_2)\in (\R^{m})^2$ and $(\vec{z}_1,\vec{z}_2)\in (\R^{m\times d})^2$;
\item[(b)] has sub-polynomial growth in $x$, uniformly with respect to $(t,\vec{y}, \vec{z})$; that is, there are $C>0$ and $q\ge 1$ such that
\begin{equation}
\label{f-bound}
|f_i(t,x,\vec{y},\vec{z})|\leq C(1+|x|_d^q),\qquad\text{for all $(t,x,\vec{y},\vec{z}) \in [0,T]\times \R^{d+m+m\times d}$}.
\end{equation}
\end{itemize}

\noindent{\bf(A2)} For $(i,j) \in \G \times \G$, the functions
\[
{g}_{ij}: (t,x)\in[0,T] \times \R^d \longmapsto g_{ij}(t,x)\in\R_+
\]
have the following properties:
\begin{itemize}
\item[(a)] let $i,j,\ell\in \G$ with $\text{card}\{i,j,\ell\}=3$, then $g_{ij}(t,x)<g_{i\ell}(t,x)+ g_{\ell j}(t,x)$, for any $(t,x) \in [0,T] \times \R^d$. Moreover, $g_{ii}(t,x)=0$;
\item[(b)] for any $i,j\in \G$, $g_{ij}$ belongs to $\cC^{1,2}([0,T]\times \R^d)$ and
\[
\rho_{ij}(t,x):=(\partial_t\,g_{ij}+ \mathbb{L}_Xg_{ij})(t,x)\le 0,\qquad\text{for all $(t,x)\in[0,T]\times\R^d$}.
\]
\end{itemize}

\begin{remark}\
\begin{enumerate}
\item Notice that condition \textbf{(A2)}-(a) implies the so-called non-free loop property which is considered in several papers including \cite{hamadene-morlais, ishi-koike}, among others. Indeed, take a loop of $\G$, i.e., a sequence $\{i_1,.....,i_{\ell}\}$ of $\G$ such that ${\ell}\geq 3$, $\text{card}\{i_1,.....,i_{\ell}\}={\ell}-1$ and $i_{\ell}=i_1$. Then, under \textbf{(A2)}-(a) we have that for any $(t,x) \in [0,T] \times \R^d$
\begin{align*}
g_{i_1i_2}&(t,x)+g_{i_2i_3}(t,x)+\ldots+g_{i_{\ell-1}i_{\ell}}(t,x)\\
&>g_{i_1i_3}(t,x)+g_{i_3i_4}(t,x)+\ldots+g_{i_{\ell-1}i_{\ell}}(t,x)>\ldots >g_{i_1i_1}(t,x)=0.
\end{align*}
\item Conditions \textbf{(A2)} are satisfied if we take, for example, $g_{ij}$ independent of $x$ and of the form $g_{ij}(t,x)=\Phi(t)|i-j|$, with $\Phi$ continuously differentiable on $[0,T]$, non-increasing and positive.
\end{enumerate}
\end{remark}

\noindent {\bf(A3)} For any $i \in \G$ the functions
\[
h_i:x\in \R^d \longmapsto h_i(x)\in \R
\]
are such that for every $x \in \R^d$
\begin{itemize}
\item[(a)] $|h_i(x)| \leq C(1 + |x|_d^p)$, for some non-negative constant $p$;
\item[(b)] $h_i(x)\geq \max_{j\neq i}(h_j(x)-g_{ij}(T,x))$.
\end{itemize}
Condition \textbf{(A3)}-(b) is usually referred to as a ``consistency condition''. This is needed in order for the process $Y$ in \eqref{main-system} below to be continuous on $[0,T]$ (provided that a solutions to \eqref{main-system} exists).

Assuming that conditions \textbf{(A0)}-\textbf{(A3)} hold, we now consider a system of reflected BSDEs with interconnected obstacles associated with
($(f_i)_{i\in \G}$,$(h_i)_{i\in \G}$,$(g_{ij})_{i,j\in \G}$). More precisely we aim at finding a $m$-tuple of $(\mathcal F_t)_{t\le T}$-adapted processes $(Y ^i,Z^i,K^i)_{i\in \G}$ which solves $\P$-a.s.~the following system: For any $i\in \G$, any $(t,x)\in [0,T]\times\R^d$, and all $s\in[t,T]$ it holds
 \begin{align}
 \label{main-system}
\left\{
\begin{array}{ll}
\displaystyle Y^i\in \textbf{S}^2_T(\R),\,\,Z^i\in \textbf{H}^2_T(\R^d)\,\, \mbox{and}\,\, K^i\in \textbf{A}^2_T(\R);\\[+4pt]
 \displaystyle     {Y}^{i}_s=h_i(X^{t,x}_T)\!+\!\!\int^T_s
\!\!f_i(r,X^{t,x}_r,({Y}_r^{k})_{k\in \G},({Z}^{k}_r)_{k\in \G})dr\!+\!{K}^{i}_T-\!{K}^{i}_s-\!\!\int^T_s\!\!{Z}^{i}_rdB_r; \\[+4pt]
\displaystyle      {Y}^{i}_s \geq \max\limits_{j\neq i}\{{Y}^{j}_s-{g}_{ij}(s,X^{t,x}_s)\};\\[+4pt]
\displaystyle \int_t^T\Big( {Y}^{i}_s-\max\limits_{j\neq i}\{{Y}^{j}_s-{g}_{ij}(s,X^{t,x}_s)\}\Big)dK^i_s=0;
\end{array}
\right.
\end{align}
where we recall that $Z^idB:=\sum^d_{j=1}Z^{ij}dB^j$ with $Z^i:=(Z^{i1}, \ldots Z^{id})$.

The rest of the paper is devoted to proving existence of a solution to \eqref{main-system}.


\section{The main result}
\label{sec:mainresult}

In this section we perform an approximation of \eqref{main-system} via a sequence of penalized problems indexed by $n\in\mathbb N$. Each penalized problem admits a solution and we are able to show that, in the limit as $n\to\infty$, we obtain a solution for \eqref{main-system}.

Given $(t,x) \in [0,T]\times \R^d$ and $n \geq 1$ we introduce a system of BSDEs whose solution is a $m$-tuple of $(\mathcal F_t)_{t\le T}$-adapted processes $(Y^{i,n;t,x},Z^{i,n;t,x})_{i \in \G}$ such that for any $i\in\G$:
\begin{align}
\label{eq:penalisedBSDE}
\left\{
\begin{array}{l}
Y^{i,n;t,x}\in \textbf{S}^2_T(\R)\:\text{and}\: Z^{i,n;t,x}\in \textbf{H}^2_T(\R^d); \\[+4pt]
\displaystyle Y^{i,n;t,x}_s  = h_i(X^{t,x}_T) + \int_s^T \Big[f_i(r, X^{t,x}_r, (Y^{k,n;t,x}_r)_{k \in \G}, (Z^{k,n;t,x}_r)_{k \in \G}) \\
\qquad\qquad \displaystyle + n\,\sum_{j\neq i} \Big(Y^{i,n;t,x}_r - Y^{j,n;t,x}_r + g_{ij}(r,X^{t,x}_r)\Big)^{-}\Big] dr  - \int_s^T Z^{i,n;t,x}_r dB_r,\\
\text{for every}\,\,s \in [t,T].
\end{array}
\right.
\end{align}

First we notice that \eqref{eq:penalisedBSDE} admits a unique solution $(Y^{i,n;t,x},Z^{i,n;t,x})_{i \in \G}$ thanks to Pardoux-Peng's result \cite{pardoux-peng}. More precisely: for any $i\in \Gamma$, the random variable $h_i(X^{t,x}_T)$ is square integrable due to \textbf{(A3)} and \eqref{estimx}; moreover, the functions
$${f}^{(n)}_i(t,x,y,z):=f_i(t,x,y,z) + n \sum_{j\neq i} \big(y_i - y_j + g_{ij}(t,x)\big)^{-}$$
are uniformly Lipschitz in $(\vec{y},\vec{z})$ by \textbf{(A1)}.
Next the Markovian nature of our setting also implies that there
exist measurable deterministic functions $(u^{i,n})_{i\in \G}$ and
$(v^{i,n})_{i\in \G}$, with $u^{i,n}:[0,T]\times\R^d\to\R$ and
$v^{i,n}:[0,T]\times\R^d\to\R^d$, such that for any $(t,x) \in
[0,T]\times \R^d$ and $s \in [t,T]$,
\begin{equation}
\label{represent1}
Y^{i,n;t,x}_s=u^{i,n}(s,X^{t,x}_s)\quad \mbox{ and } \quad Z^{i,n;t,x}_s=v^{i,n}(s,X^{t,x}_s).
\end{equation}
One can refer to \cite{karoui-peng-quenez} (Theorem 4.1, p.\ 46) for more details.
Finally, the following representation holds: for any $i\in \G$ and $(t,x) \in [0,T]\times \R^d$ one has
\begin{align}
\label{representation}
u^{i,n}(t,x)=&\E\bigg[h_i(X^{t,x}_T) + \int_t^T \Big\{f_i(r, X^{t,x}_r, (Y^{k,n;t,x}_r)_{k \in \G}, (Z^{k,n;t,x}_r)_{k \in \G}) \\
&\qquad\qquad\qquad+ n\sum_{j\neq i} \Big(Y^{i,n;t,x}_r - Y^{j,n;t,x}_r + g_{ij}(r,X^{t,x}_r)\Big)^{-}\Big\}dr\bigg].\nonumber
\end{align}

In order to simplify notation, from now and when no confusion may arise, we will drop the $(t,x)$-dependence of $(Y^{i,n;t,x},Z^{i,n;t,x})_{i\in \Gamma}$, and we will simply write $(Y^{i,n},Z^{i,n})_{i\in \Gamma}$. Moreover, we will simply denote $f_i(r, X^{t,x}_r, Y^{n}_r, Z^{n}_r)$ with the convention that $Y^{n}:=(Y^{k,n})_{k\in\G}$ and $Z^{n}:=(Z^{k,n})_{k\in\G}$ .
The next proposition provides a bound for the penalizing term in the driver of \eqref{eq:penalisedBSDE}, which is uniform with respect to $n$.
\begin{proposition}
\label{prop:penalisation}
Let $(t,x) \in [0,T] \times \R^d$ be given and fixed. Then, for $q \geq 1$ as in Assumption {\bf(A1)}-(b), there exists $C=C(q,T)>0$ such that, for any $i \in \Gamma$ and $n\geq 1$, one has
\begin{equation}
\label{eq:penalisation}
n \sum_{j\neq i} \Big(Y^{i,n}_s - Y^{j,n}_s + g_{ij}(s,X^{t,x}_s)\Big)^{-} \leq C\big(1 + |X^{t,x}_s|^{q}\big), \quad t\le s\le T.
\end{equation}
\end{proposition}

\begin{proof}
Fix $(t,x) \in [0,T] \times \R^d$, and for given $i,j \in \Gamma$ and $n \geq 1$, set
\begin{equation}
\label{eq:defxi1}
\xi^{ij,n}_{s}:=Y^{i,n}_{s} - Y^{j,n}_{s} + g_{ij}(s,X^{t,x}_{s}), \qquad s\in [t,T].
\end{equation}
By an application of It\^o-Tanaka's formula (cf.\ \cite{KS}, Chapter 3.7, Theorem 7.1), for every $s\in [t,T]$ we obtain
\begin{align}
\label{eq:stima1}
e^{-n(T-s)}\big(\xi^{ij,n}_{T})^-  =&\, \big(\xi^{ij,n}_{s})^{-} -\! \int_s^T \!{\mathds{1}}_{\{\xi^{ij,n}_{u} < 0\}} e^{-n(u-s)} d\xi^{ij,n}_{u} -  \int_s^T n e^{-n(u-s)} \big(\xi^{ij,n}_{u}\big)^- du \nonumber \\
& + \frac{1}{2}\int_s^T e^{-n(u-s)} dL^0_u(\xi^{ij,n}),
\end{align}
where $L^0(\xi^{ij,n})$ denotes the local-time at zero of the semimartingale $\xi^{ij,n}$. Noticing that the integral with respect to the local-time is nonnegative, we obtain from \eqref{eq:stima1} that for every $s\in [t,T]$
\begin{align}
\label{eq:stima2}
\big(\xi^{ij,n}_{s}\big)^{-} \leq&\, e^{-n(T-s)}\big(\xi^{ij,n}_{T})^- \\
&+\! \int_s^T \!\!{\mathds{1}}_{\{\xi^{ij,n}_{u} < 0\}} e^{-n(u-s)} d\xi^{ij,n}_{u} +\!  \int_s^T \!\!n\, e^{-n(u-s)} \big(\xi^{ij,n}_{u}\big)^- du.\nonumber
\end{align}

We now want to find a convenient expression for $d\xi^{ij,n}_{u}$. In the definition of $\xi^{ij,n}$ (cf.\ \eqref{eq:defxi1}) we may express $Y^{i,n}$ and $Y^{j,n}$ in terms of their associated BSDEs \eqref{eq:penalisedBSDE}. This gives, for any $u \in [t,T]$
\begin{align}
\label{eq:defxi}
\xi^{ij,n}_{u}  =& g_{ij}(u, X^{t,x}_u) + (h_i - h_j)(X^{t,x}_T) + \int_u^T (f_i - f_j)(r,X^{t,x}_r,Y^{n}_r,Z^{n}_r)dr \nonumber \\
& + n \sum_{k \neq i} \int_u^T \big(\xi^{ik,n}_{r}\big)^- dr - n \sum_{k \neq j} \int_u^T \big(\xi^{jk,n}_{r}\big)^- dr - \int_u^T (Z^{i,n}_r - Z^{j,n}_r) dB_r.
\end{align}
Then taking the differential with respect to the time variable $u$, and recalling $\rho_{ij}$ from {\bf(A2)}-(b), gives
\begin{align}
\label{eq:diffxi}
d\xi^{ij,n}_{u} =& \sum_{k =1}^d \frac{\partial g_{ij}}{\partial x_k}(u,X^{t,x}_u) \sigma_k(u,X^{t,x}_u)dB_u+ (Z^{i,n}_u - Z^{j,n}_u) dB_u \nonumber \\
&+\rho_{ij}(u,X^{t,x}_u)du - (f_i - f_j)(u,X^{t,x}_u,Y^{n}_u,Z^{n}_u)du \\[+4pt]
&- n \sum_{k \neq i} \big(\xi^{ik,n}_u\big)^- du + n \sum_{k \neq j} \big(\xi^{jk,n}_u\big)^- du . \nonumber
\end{align}
where we have also set $\sigma_k(u,X_u)dB_u:=\sum_{\ell}\sigma_{k\ell}(u,X_u)dB^\ell_u$ to simplify the notation. We multiply \eqref{eq:diffxi} by ${\mathds{1}}_{\{\xi^{ij,n}_{u} < 0\}} e^{-n(u-s)}$ and integrate over $[s,T]$. Then adding
\[
\int_s^T n e^{-n(u-s)} \big(\xi^{ij,n}_{u}\big)^- du
\]
we obtain
\begin{align}
\label{eq:stima3}
 \int_s^T& {\mathds{1}}_{\{\xi^{ij,n}_{u} < 0\}} e^{-n(u-s)} d\xi^{ij,n}_{u} +  \int_s^T n e^{-n(u-s)} \big(\xi^{ij,n}_{u}\big)^- du  \nonumber\\
 =&\! \int_s^T \!\!\!{\mathds{1}}_{\{\xi^{ij,n}_{u} < 0\}} e^{-n(u-s)} \Big[\rho_{ij}(u,X^{t,x}_u) - (f_i - f_j)(u,X^{t,x}_u,Y^{n}_u,Z^{n}_u)\Big] du \nonumber\\
& - n \sum_{k \neq i} \int_s^T\!\!\!  {\mathds{1}}_{\{\xi^{ij,n}_{u} < 0\}} e^{-n(u-s)} \big(\xi^{ik,n}_u\big)^- du + n \sum_{k \neq j} \int_s^T \!\!\! {\mathds{1}}_{\{\xi^{ij,n}_{u} < 0\}} e^{-n(u-s)} \big(\xi^{jk,n}_u\big)^- du   \\
& + \int_s^T n\,e^{-n(u-s)} \big(\xi^{ij,n}_u\big)^- du + M^{ij,n}_{s,T},\nonumber
\end{align}
where we have defined
\begin{equation}
\label{def:M}
M^{ij,n}_{s,T}:=\!\int_s^T \!\!{\mathds{1}}_{\{\xi^{ij,n}_{u} < 0\}} e^{-n(u-s)} \Big[\sum_{k = 1}^d\frac{\partial g_{ij}}{\partial x_k}(u,X^{t,x}_u) \sigma_{k}(u,X^{t,x}_u)dB_u + (Z^{i,n}_u - Z^{j,n}_u)dB_u\Big].
\end{equation}
Notice in particular that $(M^{ij,n}_{t,s})_{s\in[t,T]}$ is indeed a martingale.

Next we provide upper bounds for some of the terms in \eqref{eq:stima3}. First we notice that for $u\in[t,T]$ it holds
\begin{align}
\label{eq:xi1}
& {\mathds{1}}_{\{\xi^{ij,n}_{u} < 0\}}\Big(\big(\xi^{jk,n}_u\big)^- - \big(\xi^{ik,n}_u\big)^- \Big) \leq \mathds{1}_{\{\xi^{ij,n}_{u} < 0\}}\Big(\xi^{jk,n}_u  - \xi^{ik,n}_u \Big)^- \nonumber \\
& = \mathds{1}_{\{Y^{j,n}_u > Y^{i,n}_u + g_{ij}(u,X^{t,x}_u)\}} \Big(Y^{j,n}_u + g_{jk}(u,X^{t,x}_u) - Y^{i,n}_u - g_{ik}(u,X^{t,x}_u)\Big)^- \\
& \leq \mathds{1}_{\{Y^{j,n}_u > Y^{i,n}_u + g_{ij}(u,X^{t,x}_u)\}} \Big(g_{ij}(u,X^{t,x}_u) + g_{jk}(u,X^{t,x}_u) - g_{ik}(u,X^{t,x}_u)\Big)^- =0\nonumber
\end{align}
by Assumption \textbf{(A2)}-(a). Also we notice that
\begin{align}
\label{eq:xi2}
 \mathds{1}_{\{\xi^{ij,n}_{u} < 0\}} \big(\xi^{ji,n}_u\big)^- =& \mathds{1}_{\{Y^{j,n}_u > Y^{i,n}_u + g_{ij}(u,X^{t,x}_u)\}} \Big(Y^{j,n}_u - Y^{i,n}_u + g_{ji}(u,X^{t,x}_u)\Big)^- \nonumber \\
\leq& \mathds{1}_{\{Y^{j,n}_u > Y^{i,n}_u + g_{ij}(u,X^{t,x}_u)\}} \Big(g_{ij}(u,X^{t,x}_u) + g_{ji}(u,X^{t,x}_u)\Big)^- =0
\end{align}
because switching costs are non-negative.

Now, simple algebra and \eqref{eq:xi1}-\eqref{eq:xi2} give
\begin{align}
\label{eq:stima4}
& \sum_{k \neq j} \int_s^T \!\!\! \mathds{1}_{\{\xi^{ij,n}_{u} < 0\}} e^{-n(u-s)} \big(\xi^{jk,n}_u\big)^- du \nonumber\\
&\:\:-  \sum_{k \neq i} \int_s^T\!\!\!  \mathds{1}_{\{\xi^{ij,n}_{u} < 0\}} e^{-n(u-s)} \big(\xi^{ik,n}_u\big)^- du + \!\int_s^T \!\!\!e^{-n(u-s)} \big(\xi^{ij,n}_u\big)^- du\\
&=  \sum_{k \neq i,j} \int_s^T\!\!\!  \mathds{1}_{\{\xi^{ij,n}_{u} < 0\}} e^{-n(u-s)}\Big(\big(\xi^{jk,n}_u\big)^- - \big(\xi^{ik,n}_u\big)^- \Big)du \nonumber\\
&\:\:\:\:+\!  \int_s^T\!\!\! \mathds{1}_{\{\xi^{ij,n}_{u} < 0\}} e^{-n(u-s)} \big(\xi^{ji,n}_u\big)^- du \leq 0.\nonumber
\end{align}
By feeding \eqref{eq:stima4} back into \eqref{eq:stima3} we obtain
\begin{align*}
\int_s^T& \mathds{1}_{\{\xi^{ij,n}_{u} < 0\}} e^{-n(u-s)} d\xi^{ij,n}_{u} +  \int_s^T \!\!n\, e^{-n(u-s)} \big(\xi^{ij,n}_{u}\big)^- du  \\
\leq&\,  M^{ij,n}_{s,T} +\! \int_s^T\!\!\! \mathds{1}_{\{\xi^{ij,n}_{u} < 0\}} e^{-n(u-s)} \Big[\rho_{ij}(u,X^{t,x}_u) - (f_i - f_j)(u,X^{t,x}_u,Y^{n}_u,Z^{n}_u)\Big] du. \nonumber
\end{align*}
The latter may be plugged in \eqref{eq:stima2} to yield
\begin{align}
\label{eq:stima5}
\big(\xi^{ij,n}_{s})^{-}  \leq& e^{-n(T-s)}\Big(h_i(X^{t,x}_T) - h_j(X^{t,x}_T) + g_{ij}(T,X^{t,x}_T)\Big)^- + M^{ij,n}_{s,T} \nonumber\\
& + \int_s^T\!\! \mathds{1}_{\{\xi^{ij,n}_{u} < 0\}} e^{-n(u-s)}\Big[\rho_{ij}(u,X^{t,x}_u) - (f_i - f_j)(u,X^{t,x}_u,Y^{n}_u,Z^{n}_u)\Big] du,
\end{align}
for every $s \in [t,T]$.

By Assumption \textbf{(A3)}-(b) we have that $\Big(h_i(X^{t,x}_T) - h_j(X^{t,x}_T) + g_{ij}(T,X^{t,x}_T)\Big)^-=0$.
Moreover, our assumptions on the switching costs $g_{ij}$ (cf.\ Asssumption \textbf{(A2)}) and on the volatility $\sigma$ (cf.~Assumption \ref{ass:1}), imply that $\E\big[M^{ij,n}_{s,T}\big|\cF_s\big]=0$ (see \eqref{def:M}) and $\rho_{ij}(u,X_u)\le0$. Then, taking conditional expectations with respect to $\mathcal{F}_s$ in \eqref{eq:stima5}, using the sub-polynomial growth of $f_i$ and $f_j$ (cf.\ Assumption \textbf{(A1)}-(b)) and \eqref{estimx}, we obtain that
\begin{align}
\label{eq:stima6}
\big(\xi^{ij,n}_{s})^{-} \leq& \E\bigg[\int_s^T \mathds{1}_{\{\xi^{ij,n}_{u} < 0\}} e^{-n(u-s)}(f_j - f_i)(u,X^{t,x}_u,Y^{n}_u,Z^{n}_u)du\,\Big|\,\mathcal{F}_s\bigg] \nonumber \\
 \leq& \int_s^T \!\!e^{-n(u-s)} c\,\big(1 + \E\big[\sup_{s \leq r \leq u} |X_r|^{q}|\mathcal{F}_s\big]\big) du \leq \frac{c}{n}\big(1 + |X_s|^{q}\big),
\end{align}
for every $s\in[t,T]$, with $q\geq 1$ and for a constant $c=c(T,q)>0$ changing from line to line and independent of $n$.

Recalling \eqref{eq:defxi1} we then conclude that for any $(i,j) \in \Gamma \times \Gamma$ and $n\geq 1$
\[
n \Big(Y^{i,n}_s - Y^{j,n}_s + g_{ij}(s,X^{t,x}_s)\Big)^{-} \leq c\,\big(1 + |X^{t,x}_s|^{q}\big), \quad t\le s\le T.
\]
Taking the summations over all $j\neq i$ and setting $C:=(m-1)c$, we finally obtain
\[
n \sum_{j \neq i}\Big(Y^{i,n}_s - Y^{j,n}_s + g_{ij}(s,X^{t,x}_s)\Big)^{-} \leq C\big(1 + |X^{t,x}_s|^{q}\big), \quad t\le s\le T.
\]
\end{proof}

From now on we denote
\begin{equation}
\label{eq:defK} K^{i,n}_{s}:= n \sum_{j \neq i} \int_t^s
\Big(Y^{i,n}_{r} - Y^{j,n}_{r} + g_{ij}(r,X^{t,x}_{r})\Big)^- dr,
\quad s\in [t,T].
\end{equation}
Thanks to Proposition \ref{prop:penalisation} we are able prove the next uniform estimate on the solution of the penalized problem.
\begin{proposition}
\label{prop:stimaL2}
Let $(t,x) \in[0,T]\times \R^d$ be arbitrary. For any $i \in \Gamma$ and $n\geq 1$ there exist constants $C>0$ and $\rho \geq 1$ independent of $n$ such that
\begin{equation}
\label{eq:stimaL2} \E\bigg[\sup_{t \leq s\leq
T}|Y^{i,n}_s|^2 + \int_t^T |Z^{i,n}_s|^2 ds + |K^{i,n}_T|^2 \bigg]
\leq C(1 + |x|^\rho).
\end{equation}
\end{proposition}

\begin{proof}
Applying It\^o's formula and recalling \eqref{eq:penalisedBSDE} we
obtain that for every $s \in [t,T]$
\begin{align}
\label{eq:stimaL2-1} |Y^{i,n}_s|^2 + \int_s^T\! |Z^{i,n}_r|^2 dr =&
|h_i(X^{t,x}_T)|^2 + 2 \int_s^T\!\!
Y^{i,n}_r f_i(r, X^{t,x}_r, Y^{n}_r, Z^{n}_r) dr\\
 &- 2 \int_s^T\!\! Y^{i,n}_r Z^{i,n}_r dB_r + 2 \int_s^T\!\! Y^{i,n}_r dK^{i,n}_r.\nonumber
\end{align}
Taking expectations and using the sub-polynomial growth of $h_i$ and $f_i$ (cf.\ Assumptions \textbf{(A1)}-(b) and \textbf{(A3)}-(a)) we get
\begin{align}
\label{eq:stimaL2-2}
\E\bigg[|Y^{i,n}_s|^2 + \int_s^T |Z^{i,n}_r|^2 dr\bigg] \leq& c_1\Big(1 +
\E\big[|X^{t,x}_T|^{2p}\big]\Big)  + 2 c_2 \E\bigg[ \int_s^T \!\!|Y^{i,n}_r|
 \Big(1 + |X^{t,x}_r|^{q}\Big) dr\bigg]\nonumber\\
& + 2 \E\bigg[\int_s^T \!\!|Y^{i,n}_r|\, n \sum_{j \neq i}
\Big(Y^{i,n}_{r} - Y^{j,n}_{r} + g_{ij}(r,X^{t,x}_{r})\Big)^-
dr\bigg]
 \end{align}
for suitable positive constants $c_1$ and $c_2$. We now use the classical inequality
$2|a b| \leq \varepsilon |a|^2 + \frac{1}{\varepsilon}|b|^2$, for any $a,b \in \R$ and $\varepsilon>0$, the bound \eqref{eq:penalisation} (notice that $q$ therein is the same as the one in \eqref{eq:stimaL2-2}) and \eqref{estimx} to obtain
\begin{align}
\label{eq:stimaL2-3} \E\bigg[|Y^{i,n}_s|^2 + \int_s^T
|Z^{i,n}_r|^2 dr\bigg]
\leq & C\Big( 1 + \E\big[|X^{t,x}_T|^{2p}\big] + \E\bigg[\int_s^T\!\! |Y^{i,n}_r|^2 dr+
\int_s^T\!\! |X^{t,x}_r|^{2q}dr\bigg]\Big)\nonumber\\
\leq& C\Big(1 + |x|^{\rho} +  \E\bigg[\int_s^T |Y^{i,n}_r|^2
dr\bigg] \Big),
\end{align}
where $\rho=2(p\vee q)$ and $C=C(T,p,q,\varepsilon)>0$ varies from line to line and it is independent of $n$.

From \eqref{eq:stimaL2-3} and Gronwall's inequality we find $\forall
s\in [t,T]$
\begin{equation}
\label{eq:stimaL2-3bis} \E\big[|Y^{i,n}_s|^2 \big] \leq
C\big(1 + |x|^{\rho}\big),
\end{equation}
for all $n \geq 1$. Letting now $c>0$ be a constant varying from line to line but independent of $n$, using \eqref{eq:stimaL2-3}, \eqref{eq:stimaL2-3bis} and Proposition \ref{prop:penalisation}, we also get
\begin{equation}
\label{eq:stimaL2-3tris} \E\bigg[\int_t^T |Z^{i,n}_r|^2 ds +
|K^{i,n}_T|^2\bigg] \leq c\big(1 + |x|^{\rho}\big).
\end{equation}
The latter and \eqref{eq:stimaL2-3bis} then yield: for any $s\in
[t,T] \times \R^d$,
\begin{equation}
\label{eq:stimaL2-4} \E\bigg[|Y^{i,n}_s|^2 + \int_t^T
|Z^{i,n}_r|^2 dr + |K^{i,n}_T|^2\bigg] \leq c\big(1 +
|x|^{\rho}\big).
\end{equation}

In order to take the supremum of the process $Y^{i,n}$ inside the expectation we need a further bound for $\sup_{t\leq s \leq T}|Y^{i,n}_s|^2$. This can be obtained by using the expression \eqref{eq:penalisedBSDE} for $Y^{i,n}$ together with the sub-polynomial growth of $f_i$ and \eqref{eq:penalisation}, that is
\begin{align}
\label{eq:stimaL2-5}
 \sup_{t\leq s \leq T}|Y^{i,n}_s|^2  \leq& 4\Big(|h_i(X^{t,x}_T)|^2 +\!\!
 \int_t^T \!\!|f_i(r, X^{t,x}_r, Y^{n}_r, Z^{n}_r)|^2 ds \nonumber\\
 &\hspace{+65pt} + |K^{i,n}_T|^2 + \sup_{t\leq s \leq T}\Big|\int_s^T \!\!Z^{i,n}_r dB_r
 \Big|^2\Big)\nonumber\\
 \le &C\Big(1+\sup_{t\le s\le T}|X^{t,x}_s|^\rho+\sup_{t\leq s \leq T}\Big|\int_s^T \!\!Z^{i,n}_r dB_r
 \Big|^2\Big).
\end{align}
Taking the expected value, applying Burkholder-Davis-Gundy's inequality and \eqref{estimx} we finally obtain \eqref{eq:stimaL2}.
\end{proof}

Recall that for each $n\ge0$ we have $Y^{i,n;t,x}_s=u^{i,n}(s,X^{t,x}_s)$ (see \eqref{represent1} and \eqref{representation}). Next we show that the sequences $(u^{i,n})_{n\ge0}$ with $i\in\G$ admit a converging subsequence.
\begin{proposition}
\label{prop:Cauchy-0}
There exists a subsequence $(n_j)_{j\geq 0}$ with $n_j\to\infty$ as $j\to\infty$, and
measurable functions $u^{i}:[0,T]\times\R^d\to\R$, $i\in\G$, such that
\begin{align}\label{eq:lim-u}
\lim_{j\to\infty}u^{i,n_j}(t,x)=u^{i}(t,x)\quad\text{for all $i\in \Gamma$ and $(t,x) \in [0,T] \times \R^d$.}
\end{align}
Moreover there exist two constants $C>0$ and $\rho\ge 1$ (independent of $n_j$) such that for any
$i\in \Gamma$ and $j\ge 0$
\begin{align}
\label{eq:estim-u0}|u^{i,n_j}(t,x)|\leq C(1+|x|^\rho ),\quad \forall (t,x)\in    [0,T] \times\R^d
\end{align}
and therefore 
\begin{align}
\label{eq:estim-u}|u^i(t,x)|\leq C(1+|x|^\rho ),\quad \forall (t,x)\in    [0,T] \times\R^d.
\end{align}

\end{proposition}
\begin{proof}
The proof is given in two steps.
\medskip

\emph{Step 1.} Let $x_0\in \R^d$ be given and fixed as in {\bf (A0)}. Consider the solution of \eqref{eq:penalisedBSDE} for $(t,x)=(0,x_0)$. By the sub-polynomial growth of $f_i$ (see \eqref{f-bound}), by \eqref{estimx} and \eqref{eq:stimaL2}, we can find $C=C(x_0)>0$ (independent of $n$ and $i\in\G$) such that
\begin{align*}
\E\bigg[\int_0^T \Big|&\,f_i(r, X^{0,x_0}_r, Y^{n;0,x_0}_r, Z^{n;0,x_0}_r) \\
& + n\sum_{\ell\neq i} \Big(Y^{i,n;0,x_0}_r - Y^{\ell,n;0,x_0}_r + g_{i\ell}(r,X^{0,x_0}_r)\Big)^{-}\Big|^2 dr\bigg]\leq C.
\end{align*}
Using the representations \eqref{represent1} for $Y^{i,n;0,x_0}$ and $Z^{i,n;0,x_0}$, the above bound reads
\begin{align}
\label{L2bound-1}
\E\bigg[\int_0^T \Big|&\,f_i(r, X^{0,x_0}_r, (u^{k,n}(r,X_r^{0,x_0}))_{k \in \G}, (v^{k,n}(r,X_r^{0,x_0}))_{k \in \G}) \\
& + n\sum_{\ell\neq i} \Big(u^{i,n}(r,X_r^{0,x_0}) - u^{\ell,n}(r,X_r^{0,x_0}) + g_{i\ell}(r,X^{0,x_0}_r)\Big)^{-}\Big|^2 dr\bigg]\leq C. \nonumber
\end{align}
For simplicity we again set $u^{n}(\,\cdot\,):=(u^{k,n}(\,\cdot\,))_{k \in \G}$ and $v^n(\,\cdot\,):=(v^{k,n}(\,\cdot\,))_{k \in \G}$ inside the functions $f_i$, when no confusion may arise.

We can express the expectation in \eqref{L2bound-1} as an integral with respect to the law of $X^{0,x_0}_r$, $r\leq T$. This gives
\begin{align*}
\int_0^T\int_{\R^d}\Big|&\,f_i(r, y,u^{n}(r,y), v^{n}(r,y)) \\
&+ n\sum_{\ell\neq i} \Big(u^{i,n}(r,y) - u^{\ell,n}(r,y) + g_{i\ell}(r,y)\Big)^{-}\Big|^2\mu(0,x_0;r,dy)dr\leq C.
\end{align*}
If we now set
\[
F^{i}_n(r,y):=f_i(r, y,u^{n}(r,y), v^{n}(r,y)) +n\sum_{\ell\neq i} \Big(u^{i,n}(r,y) - u^{\ell,n}(r,y) + g_{i\ell}(r,y)\Big)^{-}
\]
we have that the map $F_n:=(F^{i}_n)_{i\in\G}$, $F_n:[0,T]\times\R^d\to\R^m$ has all its components bounded in $L^2([0,T]\times \R^d, \mu(0,x_0;r,dy)dr)$ uniformly with respect to $n$. Therefore, the sequence $(F_n)_{n\ge 0}$ admits a subsequence $(F_{n_j})_{j\geq 0}$ such that $F^{i}_{n_j}\to F_i$
weakly in $L^2([0,T]\times \R^d, \mu(0,x_0;r,dy)dr)$ as $j\to\infty$, for each $i \in \G$. Notice that the subsequence may depend on $x_0$.
\medskip

\emph{Step 2.} Here we want to prove that \eqref{eq:lim-u} holds along the subsequence $(n_j)_{j\ge0}$ found above. In particular, given $(t,x) \in [0,T] \times \R^d$ we will prove that the sequence $(u^{i,n_j}(t,x))_{j\geq 0}$ is of Cauchy type.

Let $\delta>0$ and $N>0$ be two constants (which will be taken small and large, respectively), and notice that by \eqref{representation} we have, for any non-negative $j,k$
\begin{align}
u^{i,n_j}(t,x)-u^{i,n_k}(t,x)=&\E\bigg[\int_t^{t+\delta} (F^{i}_{n_j}(r,X^{t,x}_r)-F^{i}_{n_k}(r,X^{t,x}_r))dr\bigg]\\
&+\E\bigg[\int_{t+\delta}^T(F^{i}_{n_j}(r,X^{t,x}_r)-F^{i}_{n_k}(r,X^{t,x}_r))1_{\{|X^{t,x}_r|\leq N\}}dr\bigg]\nonumber\\
&+\E\bigg[\int_{t+\delta}^T(F^{i}_{n_j}(r,X^{t,x}_r)-F^{i}_{n_k}(r,X^{t,x}_r))1_{\{|X^{t,x}_r|> N\}} dr\bigg]\nonumber\\
&=:\Theta_1^{jk}+\Theta_2^{jk}+\Theta_3^{jk}.\nonumber
\end{align}

In what follows we let $C=C(t,x)>0$ be a suitable constant (i.e.~sufficiently large for our purposes) independent of $\delta$ and $N$. Due to \eqref{f-bound} and \eqref{eq:penalisation} we easily get $|\Theta_1^{jk}|\leq C\cdot \delta$. Moreover, the bounds in \eqref{f-bound} and \eqref{eq:penalisation}, together with Cauchy-Schwarz and Markov inequalities yield $|\Theta_3^{jk}|\leq C/N$.
Now we use the law of $X^{t,x}$ to rewrite $\Theta_2^{jk}$ as
\begin{align*}
\Theta_2^{jk}=&\E\bigg[\int_{t+\delta}^T(F^{i}_{n_j}(r,X^{t,x}_r)-F^{i}_{n_k}(r,X^{t,x}_r))1_{\{|X^{t,x}_r|\leq N\}} dr\bigg]\\
=&\int_{t+\delta}^T\int_{\R^d}(F^{i}_{n_j}(r,y)-F^{i}_{n_k}(r,y))1_{\{|y|\leq N\}}\mu(t,x;r,dy) dr.\nonumber
\end{align*}
The $L^2$-domination condition \textbf{(A0)} implies
\begin{align}
\Theta_2^{jk}=\int_{t+\delta}^T\int_{\R^d}(F^{i}_{n_j}(r,y)-F^{i}_{n_k}(r,y))1_{\{|y|\leq N\}}\phi_{t,x,x_0}^{\delta}(r,y)\mu(0,x_0;r,dy)dr.
\end{align}
By assumption $\phi_{t,x,x_0}^{\delta}\in L^{2}([t+\delta,T]\times [-N,N]^d;\ \mu(0,x_0;r,dy)dr)$, hence weak convergence of the sequence $(F^{i}_{n_j})_{j\geq 0}$ implies $\limsup_{j,k\rightarrow \infty}|\Theta_2^{jk}|=0$.

Collecting the estimates for $\Theta_1^{jk}$, $\Theta_2^{jk}$ and $\Theta_3^{jk}$ we obtain
\[
\limsup_{j,k\rightarrow \infty}|u^{i,n_j}(t,x)-u^{i,n_k}(t,x)|\leq
C(\delta+ N^{-1})
\]
and, letting $\delta\to 0$ and $N\to\infty$, we complete the proof of \eqref{eq:lim-u}. Finally, estimates \eqref{eq:estim-u0} and \eqref{eq:estim-u} follow by using the representation formula \eqref{represent1} in \eqref{eq:stimaL2-3bis}, with $s=t$, and thanks to \eqref{eq:lim-u}.
\end{proof}

As a byproduct of the previous result we have the following.
\begin{corollary}
\label{cor:Cauchy} For any $i\in \G$ one has
\begin{equation}
\label{eq:Cauchy} \lim_{j,k \rightarrow
\infty}\E\bigg[\int_t^T |Y^{i,n_j}_s - Y^{i,n_k}_s|^2 ds +
\int_t^T |Z^{i,n_j}_s - Z^{i,n_k}_s|^2 ds\bigg]=0.
\end{equation}
\end{corollary}
\begin{proof} Convergence of the first term in \eqref{eq:Cauchy} follows from the convergence result \eqref{eq:lim-u} and by using the dominated convergence theorem, which is enabled by \eqref{eq:estim-u0} and \eqref{estimx}. Convergence of the second term in
\eqref{eq:Cauchy} is obtained in a classical way. By It\^o's formula
and using the same estimates as in the proof of Proposition
\ref{prop:stimaL2} (and Lipschitz continuity of $f_i$) we get
\begin{align*}
\E&\left[\int_t^T|Z^{i,n_j}_s-Z^{i,n_k}_s|^2\right]\\
\le & 2c\,\varepsilon\,\E\left[\int_t^T\!\!\big(Y^{i,n_j}_s-Y^{i,n_k}_s\big)^2ds\right]+\frac{2c}{\varepsilon}\sum_{\alpha\in\G}\E\left[\int_t^T\!\!\big(|Y^{\alpha,n_j}_s-Y^{\alpha,n_k}_s|^2+|Z^{\alpha,n_j}_s-Z^{\alpha,n_k}_s|^2\big)ds\right]\\
&+2c\,\E\left[\int_t^T|Y^{i,n_j}_s-Y^{i,n_k}_s|(1+|X^{t,x}_s|^q)ds\right],
\end{align*}
for a suitable constant $c>0$ and arbitrary $\varepsilon>0$. Taking the summation over $i\in\G$ and picking $\varepsilon$ sufficiently large we may conclude that
\begin{align*}
\sum_{\alpha\in\G}\E&\left[\int_t^T|Z^{\alpha,n_j}_s-Z^{\alpha,n_k}_s|^2\right]\\
\le& c_\varepsilon\sum_{\alpha\in\G}\E\left[\int_t^T\!\!
|Y^{\alpha,n_j}_s-Y^{\alpha,n_k}_s|^2ds+\int_t^T\!\!|Y^{\alpha,n_j}_s-Y^{\alpha,n_k}_s|(1+|X^{t,x}_s|^q)ds\right]
\end{align*}
where $c_\varepsilon>0$ depends on $\varepsilon$ but is independent of $j,k$. Hence taking limits as $j,k\to\infty$ and using the above result we finally obtain \eqref{eq:Cauchy}.
\end{proof}

We can now prove the main result of this paper, which establishes the existence of a solution to system \eqref{main-system}. In what follows the subsequence $(n_j)_{j\ge 0}$ is the same as the one in Proposition \ref{prop:Cauchy-0}.
\begin{theorem}
\label{thm:main}
There exists a solution $(Y^i,Z^i,K^i)_{i \in \Gamma}$ to {\color{red}{\eqref{main-system}}}. Moreover, for any $i\in \Gamma$ and $t\in[0,T]$ it holds
\begin{equation}
\label{eq:convergence}
\lim_{j \rightarrow \infty}\E\bigg[\sup_{t\leq s \leq T}|Y^{i,n_j}_s - Y^{i}_s|^2 + \int_t^T |Z^{i,n_j}_s - Z^{i}_s|^2 ds + \sup_{t\leq s \leq T}|K^{i,n_j}_s - K^{i}_s|^2\bigg]=0.
\end{equation}
\end{theorem}
\begin{proof}
The proof is given in two steps. We first prove, in step 1, that
there exists $(Y^i,Z^i,K^i)_{i \in \Gamma}$ satisfying the first
equation in \eqref{main-system} and such that \eqref{eq:convergence}
holds. Then we prove, in step 2, that $(Y^i,Z^i,K^i)_{i \in \Gamma}$
fulfils the second and third conditions in \eqref{main-system} as
well.
\smallskip

\noindent \emph{Step 1.} Let $(t,x) \in [0,T] \times \R^d$ be fixed.
For any $i\in \Gamma$ let us set:
\begin{itemize}
\item[ i)] $Y^{i}_s=u^i(s,X^{t,x}_s), \,\,s\in [t,T]$, with $u^i$ as in \eqref{eq:lim-u};
\item[ii)] $(Z^i_s)_{s\in [t,T]}$ the limit in $\textbf{H}^2_T(\R^d)$ of
$(Z^{i,n_j}_s)_{s\in [t,T]}$ which exists thanks to
\eqref{eq:Cauchy}.
\end{itemize}

\noindent It is clear that 
\[
Y^i_s=u^i(s,X^{t,x}_s)=\lim_{j\to\infty}u^{i,n_j}(s,X^{t,x}_s)=\lim_{j\to\infty}Y^{i,n_j}_s\quad\text{$\P$-a.s.~$\forall s\in[t,T]$}
\]
Let us now show that for any $i\in \Gamma$, the sequence
$(Y^{i,n_j})_{j\ge 0}$ is Cauchy in $\textbf{S}^2_T(\R)$ so that it converges to $Y^i$ in $\textbf{S}^2_T(\R)$. 
By using It\^o's formula, Lipschitz property of $f_i$ and the bound
in Proposition \ref{prop:penalisation}, we can argue in a similar
way to the proof of Proposition \ref{prop:stimaL2} and obtain for
all $u\in[t,T]$
\begin{align}\label{eq:exist}
\sup_{t\le u\le T}&|Y^{i,n_j}_u-Y^{i,n_k}_u|^2+\int_t^T|Z^{i,n_j}_s-Z^{i,n_k}_s|^2\nonumber\\
\le & 2c\,\varepsilon\int_t^T\!\!\big(Y^{i,n_j}_s-Y^{i,n_k}_s\big)^2ds\nonumber\\
&+\frac{2c}{\varepsilon}\sum_{\alpha\in\G}\int_t^T\!\!\big(|Y^{\alpha,n_j}_s-Y^{\alpha,n_k}_s|^2+|Z^{\alpha,n_j}_s-Z^{\alpha,n_k}_s|^2\big)ds\nonumber\\
&+2c\int_t^T|Y^{i,n_j}_s-Y^{i,n_k}_s|(1+|X^{t,x}_s|^q)ds\\
&+\sup_{t\le u\le T}\bigg|\int_u^T(Y^{i,n_j}_s-Y^{i,n_k}_s)(Z^{i,n_j}_s-Z^{i,n_k}_s)dB_s\bigg|,\nonumber
\end{align}
where $c>0$ is a suitable constant independent of $j,k$ and $\varepsilon>0$ is also arbitrary. Notice that by Burkholder-Davis-Gundy's inequality and $|ab|\le \varepsilon|a|^2+\varepsilon^{-1}|b|^2$ we have
\begin{align}\label{eq:exist1}
\E&\left[\sup_{t\le u\le T}\bigg|\int_u^T(Y^{i,n_j}_s-Y^{i,n_k}_s)(Z^{i,n_j}_s-Z^{i,n_k}_s)dB_s\bigg|\right]\nonumber\\
\le& C\E\left[\bigg(\int_t^T(Y^{i,n_j}_s-Y^{i,n_k}_s)^2(Z^{i,n_j}_s-Z^{i,n_k}_s)^2ds\bigg)^{\frac{1}{2}}\right]\\
\le& C\E\left[\sup_{t\le u\le T}|Y^{i,n_j}_u-Y^{i,n_k}_u|\bigg(\int_t^T(Z^{i,n_j}_s-Z^{i,n_k}_s)^2ds\bigg)^{\frac{1}{2}}\right]\nonumber\\
\le& C\varepsilon\E\left[\sup_{t\le u\le T}|Y^{i,n_j}_u-Y^{i,n_k}_u|^2\right]+\frac{C}{\varepsilon}\E\left[\int_t^T(Z^{i,n_j}_s-Z^{i,n_k}_s)^2ds\right],\nonumber
\end{align}
for a suitable $C>0$ independent of $j,k$ and any $\varepsilon>0$. Taking expectations in \eqref{eq:exist} and using \eqref{eq:exist1} (with $\varepsilon<1/C$), after rearranging terms we then obtain
\begin{align}\label{eq:exist2}
\E\bigg[\sup_{t\le u\le T}|Y^{i,n_j}_u-Y^{i,n_k}_u|^2\bigg]\le &\, c_{\varepsilon}\E\bigg[\sum_{\alpha\in\G}\int_t^T\!\!\big(|Y^{\alpha,n_j}_s-Y^{\alpha,n_k}_s|^2+|Z^{\alpha,n_j}_s-Z^{\alpha,n_k}_s|^2\big)ds\bigg]\nonumber\\
&+2c\,\E\bigg[\int_t^T|Y^{i,n_j}_s-Y^{i,n_k}_s|(1+|X^{t,x}_s|^q)ds\bigg],
\end{align}
where $c_\varepsilon>0$ may depend on $\varepsilon>0$ but is independent of $j,k$.
Letting now $j,k\to\infty$ and using Corollary \ref{cor:Cauchy} we obtain that $Y^{i,n_j}$
forms a Cauchy sequence in $\textbf{S}^2_T(\R)$ as claimed.
\vspace{+3pt}

Let us now define $K^i$, $i\in \Gamma$, as:
\begin{align}\label{eq:limK}
K^i_s:=Y^{i}_t-Y^{i}_s-\int_t^s\!\!f_i(u,X^{t,x}_u,(Y^{k}_u)_{k\in\G},(Z^{k}_u)_{k\in\G})du
+\int_t^s\!\!Z^{i}_udB_u,\,\,s\in [t,T]. 
\end{align}
Since $Y^{i,n_j}$ converges in $\textbf{S}^2_T(\R)$, and upon recalling Lipschitz property of $f_i$ and \eqref{eq:Cauchy}, it is easy to verify that $K^i$ is the limit in $\textbf{S}^2_T(\R)$ of the sequence $(K^{i,n_j})_{j\ge0}$ defined by (see \eqref{eq:defK} and
\eqref{eq:penalisedBSDE})
\[
K^{i,n_j}_s=Y^{i,n_j}_t-Y^{i,n_j}_s-\int_t^s\!\!f_i(u,X^{t,x}_u,(Y^{k,n_j}_u)_{k\in\G},(Z^{k,n_j}_u)_{k\in\G})du
+\int_t^s\!\!Z^{i,n_j}_udB_u,\,\,s\in [t,T].
\] 
Hence \eqref{eq:convergence} holds and, by
\eqref{eq:limK}, $(Y^i,Z^i,K^i)_{i\in \Gamma}$ verify the first
equation of \eqref{main-system}.
\smallskip

\noindent \emph{Step 2.} It only remains to show that the second and
third conditions in \eqref{main-system} are satisfied by
$(Y^i,Z^i,K^i)_{i \in \Gamma}$. Proposition \ref{prop:penalisation}
implies that there exists $C>0$ for which
\begin{equation}
\label{faltoff1} \E\bigg[\int_t^T \sum_{j\neq i}
\big(Y^{i,n}_s - Y^{j,n}_s + g_{ij}(s,X^{t,x}_s)\big)^{-} ds \bigg]
\leq \frac{C}{n}.
\end{equation}
Using \eqref{eq:Cauchy} and letting $n \uparrow \infty$ (along the subsequence used in \eqref{eq:Cauchy}) we immediately obtain
\begin{equation}
\label{flatoff2} \E\bigg[\int_t^T \sum_{j\neq i}\big(Y^{i}_s
- Y^{j}_s + g_{ij}(s,X^{t,x}_s)\big)^{-} ds\bigg] =0.
\end{equation}
Hence, for all $i,j\in \Gamma$, $Y^{i}_s \geq Y^{j}_s +
g_{ij}(s,X^{t,x}_s)$, $\P$-a.s.\ for every~$s \in [t,T]$ (recall
that $s\mapsto Y^k_s$, $k\in\G$ is indeed continuous as uniform
limit of continuous processes). In particular
\begin{equation}
\label{flatoff3} Y^{i}_s \geq \max_{j\neq i}\Big(Y^{j}_s -
g_{ij}(s,X^{t,x}_s)\Big), \quad \P-\text{a.s.} \quad \forall s \in
[t,T].
\end{equation}

Thanks to \eqref{eq:convergence}, by Tchebyshev's inequality we have that for any $i\in \Gamma$,
\begin{align}\label{eq:limP}
\lim_{j\to\infty}\P\left(\sup_{t\le s\le
T}\big(|Y^{i,n_j}_s-Y^{i}_s|+|K^{i,n_j}_s-K^{i}_s|\big)\ge
\varepsilon\right)=0
\end{align}
for any $\varepsilon>0$. Moreover, for a.e.~$\omega\in\Omega$ and
for each $j\ge 0$ the map $s\mapsto K^{i,n_j}_s(\omega)$ is
increasing and continuous, hence it is a (random) continuous measure
on $[t,T]$. The same holds for the limit process $K^i$. The uniform
convergence in \eqref{eq:limP} implies that (up to selecting a
subsequence) $K^{i,n_j}(\omega)\to K^{i}(\omega)$ as $j\to\infty$
\emph{in general} in the sense of measures (see \cite[Ch.~3]{S}).
Therefore, for $\P$-a.e.~$\omega\in\Omega$, it holds $d
K^{i,n_j}(\omega)\to d K^{i}(\omega)$ weakly as $j\to\infty$ (see
\cite[Thm.~1, Ch.~3]{S}) and
\begin{align}
\label{eq:flat1}
\lim_{j\to\infty}&\int_t^T \!\!\big[Y^{i,n_j}_s - \max_{k\neq i}\big(Y^{k,n_j}_s - g_{ik}(s,X^{t,x}_s)\big)\big]dK^{i,n_j}_s\nonumber\\
&\quad\quad\quad=\int_t^T \!\!\big[Y^{i}_s - \max_{k\neq
i}\big(Y^{k}_s - g_{ik}(s,X^{t,x}_s)\big)\big] dK^{i}_s,
\quad\P-a.s.
\end{align}
We now notice that the left-hand side of \eqref{eq:flat1} is
non-positive due to \eqref{eq:defK} and the fact that for any $\ell\neq i$ and all
$s\in[t,T]$
\[
\Big(Y^{i,n_j}_s - \max_{k\neq i}\big(Y^{k,n_j}_s - g_{ik}(s,X^{t,x}_s)\big)\Big)\Big(Y^{i,n_j}_s - Y^{\ell,n_j}_s + g_{i\ell}(s,X^{t,x}_s)\Big)^-\le0,\qquad \P-a.s.
\]
However, the right-hand side of \eqref{eq:flat1} is non-negative due to \eqref{flatoff3} and the fact that $K^i$ is increasing. Hence we get
\[
\int_t^T \!\!\big[Y^{i}_s - \max_{k\neq i}\big(Y^{k}_s -
g_{ik}(s,X^{t,x}_s)\big)\big] dK^{i}_s=0, \quad\P-a.s.,
\]
which completes the proof.
\end{proof}

We now provide a corollary of Proposition \ref{prop:Cauchy-0} and Theorem \ref{thm:main}.
\begin{corollary}\label{cor:uv}
There exist measurable deterministic functions $(u^i)_{i\in\G}$ and $(v^i)_{i\in\G}$ with $u^i:[0,T]\times\R^d\to\R$ and $v^i:[0,T]\times\R^d\to\R^d$ such that for any $(t,x)\in[0,T]\times\R^d$
\begin{align}
Y^{i;t,x}_s=u^i(s,X^{t,x}_s)\quad\text{and}\quad Z^{i;t,x}_s=v^i(s,X^{t,x}_s),\quad\P-a.s.~\text{for a.e.}~s\in[t,T].
\end{align}
\end{corollary}
\begin{proof}
It only remains to show the existence of $v^i$. Recall $v^{i,n}$ from \eqref{represent1} and set $v^i:=\limsup_{j\to\infty}v^{i,n_j}$, where the limit is taken along the subsequence introduced in Proposition \ref{prop:Cauchy-0}. Then, using that $Z^{i,n_j}_s\to Z^{i}_s$, $\P$-a.s.~for a.e.~$s\in[t,T]$, and choosing $(s,\omega)$ such that the convergence indeed holds we find
\[
v^i(s,X_s(\omega))=\limsup_{j\to\infty}v^{i,n_j}(s,X_s(\omega))=\limsup_{j\to\infty}Z^{i,n_{j}}_s(\omega)=Z^i_s(\omega).
\]
\end{proof}
\medskip

\textbf{Acknowledgement}: The authors acknowledge hospitality by
Center for Mathematical Economics (IMW) at Bielefeld University,
D\'epartement de Math\'ematiques at Le Mans Universit\'e and School
of Mathematics at University of Leeds. The first and third named
authors were partially supported by the Heilbronn Institute and the
School of Mathematics at Leeds during the research week
``\emph{Optimal stopping games for ambiguity-averse players}'' (4-9
September 2017). Financial support by the German Research Foundation
(DFG) through the Collaborative Research Centre `Taming uncertainty
and profiting from randomness and low regularity in analysis,
stochastics and their applications' is gratefully acknowledged by
the second named author.


\end{document}